\documentclass[12pt,reqno]{article}

\usepackage[usenames]{color}
\usepackage{amssymb}
\usepackage{amsmath}
\usepackage{amsthm}
\usepackage{amsfonts}
\usepackage{amscd}
\usepackage{graphicx}

\usepackage[colorlinks=true,
linkcolor=webgreen,
filecolor=webbrown,
citecolor=webgreen]{hyperref}

\definecolor{webgreen}{rgb}{0,.5,0}
\definecolor{webbrown}{rgb}{.6,0,0}

\usepackage{color}
\usepackage{fullpage}
\usepackage{float}

\usepackage{graphics}
\usepackage{latexsym}

\begin{document}

\theoremstyle{plain}
\newtheorem{theorem}{Theorem}
\newtheorem{corollary}[theorem]{Corollary}
\newtheorem{lemma}{Lemma}
\newtheorem{example}{Examples}
\newtheorem*{remark}{Remark}

\begin{center}
\vskip 1cm
{\LARGE\bf
Series associated with a forgotten \\

\vskip .3cm

identity of N\"orlund}

\vskip 1cm

{\large
Kunle Adegoke \\
Department of Physics and Engineering Physics, \\ Obafemi Awolowo University, 220005 Ile-Ife \\ Nigeria \\
\href{mailto:adegoke00@gmail.com}{\tt adegoke00@gmail.com}

\vskip 0.2 in

Robert Frontczak \\
Independent Researcher, 72762 Reutlingen \\ Germany \\
\href{mailto:robert.frontczak@web.de}{\tt robert.frontczak@web.de}
}

\end{center}

\vskip .2 in

\begin{abstract}
We apply a seemingly forgotten series expression of N\"orlund for the psi function to express
infinite series involving inverse factorials in closed form. Many of such series 
contain products of Catalan numbers and (odd) harmonic numbers. We also prove some new series for $\pi$.
\end{abstract}



\bigskip

\section{Motivation and Preliminaries}

The Gamma function, $\Gamma(z)$, is defined for $\Re(z)>0$ by the integral \cite{Srivastava}
\begin{equation*}
\Gamma (z) = \int_0^\infty e^{- t} t^{z - 1}\,dt.
\end{equation*}
The function $\Gamma(z)$ has the property that it extends the classical factorial function to the complex plane
by $(z-1)!=\Gamma(z)$. It can be extended to the whole complex plane by analytic continuation.
Closely related to the Gamma function is the psi (or digamma) function defined by $\psi(z)=\Gamma'(z)/\Gamma(z)$.
It possesses a range of integral representations and the infinite series expressions.
Two of such series expressions are \cite[p. 14]{Srivastava}
\begin{equation}\label{psi_expr1}
\psi (z) = - \gamma + \sum_{k=0}^\infty \left ( \frac{1}{k+1}-\frac{1}{k+z} \right ),
\end{equation}
and
\begin{equation}\label{psi_expr2}
\psi (z+n) = \psi(z) + \sum_{k=1}^n \frac{1}{z+k-1},
\end{equation}
where in \eqref{psi_expr1} $\gamma$ denotes the Euler-Mascheroni constant. Coffey \cite{Coffey1} has obtained a variety 
of other series and integral representations of $\psi(z)$. Summations over digamma and polygamma functions were studied, among others,
by Alzer et al. \cite{Alzer} and Coffey \cite{Coffey2}. \\

A series expression that seems to be not well known is the following expression
\begin{equation}\label{psi_expr3}
\psi (z+h) - \psi(z) = \sum_{n=0}^\infty \frac{(-1)^n\,h\,(h-1)\cdots (h-n)}{(n+1)\,z\,(z+1)\cdots (z+n)},
\end{equation}
which holds for all $z,h\in\mathbb{C}$ with $\Re(z)>0$ and $\Re(z+h)>0$. Identity \eqref{psi_expr3}
is stated by Hille in his paper \cite{Hille} from 1927 and is attributed to N\"orlund. \\

The seemingly forgotten identity \eqref{psi_expr3} will play the key role in this paper. The common feature of
all series studied here is the appearance of inverse factorials. Similar series were studied by Sofo \cite{Sofo1,Sofo2} 
and more recently by Karp and Prilepkina \cite{Karp}, and also by Boyadzhiev in his papers from 2020 and 2023 
\cite{Boyadzhiev1,Boyadzhiev2}. A finite variant was the subject of another paper by Sofo \cite{Sofo3}. \\

Among other things we will establish connections with harmonic numbers $H_\alpha$ and odd harmonic numbers $O_\alpha$ defined
for $0\ne\alpha\in\mathbb C\setminus\mathbb Z^{-}$ by the recurrence relations
\begin{equation*}
H_\alpha = H_{\alpha - 1} + \frac{1}{\alpha} \qquad \text{and} \qquad O_\alpha = O_{\alpha - 1} + \frac{1}{2\alpha - 1},
\end{equation*}
with $H_0=0$ and $O_0=0$. Harmonic numbers are connected to the digamma function through the fundamental relation
\begin{equation}\label{Harm_psi}
H_\alpha = \psi(\alpha + 1) + \gamma.
\end{equation}

Generalized harmonic numbers $H_\alpha^{(m)}$ and generalized odd harmonic numbers $O_\alpha^{(m)}$ of order $m\in\mathbb C$ 
are defined by
\begin{equation*}
H_\alpha^{(m)} = H_{\alpha - 1}^{(m)} + \frac{1}{\alpha^m} \qquad \text{and} \qquad
O_\alpha^{(m)} = O_{\alpha - 1}^{(m)} + \frac{1}{(2\alpha - 1)^m},
\end{equation*}
with $H_0^{(m)}=0$ and $O_0^{(m)}=0$ so that $H_\alpha=H_\alpha^{(1)}$ and $O_\alpha=O_\alpha^{(1)}$.

The recurrence relations imply that if $\alpha=n$ is a non-negative integer, then
\begin{equation*}
H_n^{(m)} = \sum_{j = 1}^n \frac{1}{j^m} \qquad \text{and} \qquad O_n^{(m)} = \sum_{j = 1}^n \frac{1}{(2j - 1)^m}.
\end{equation*}

Generalized harmonic numbers are linked to the polygamma functions $\psi^{(r)}(z)$ of order $r$ defined by
\begin{equation*}
\psi^{(r)} (z) = \frac{d^r}{dz^r}\psi(z) = (- 1)^{r + 1} r!\sum_{j = 0}^\infty \frac{1}{{( j + z )^{r + 1} }},
\end{equation*}
through
\begin{equation*}
H_z^{(r)} = \zeta (r) + \frac{(- 1)^{r - 1}}{(r - 1)!} \psi^{(r - 1)} (z + 1),
\end{equation*}
where $\zeta(y)$ is the Riemann zeta function.

To whet the reader's appetite for reading on, we present the following samples from our results:
\begin{equation*}
1 + \sum_{n=1}^\infty \frac{1}{n+1} \prod_{j=1}^n \frac{3j-1}{3j+1} = \frac{\pi}{\sqrt{3}},
\end{equation*}
\begin{equation*}
1 + \sum_{n=1}^\infty \frac{2^n}{n+1} \prod_{j=1}^n \frac{5j-2}{10j+3} = \frac{3\pi}{4}\sqrt{\frac{\sqrt{5}}{\alpha^3}},
\qquad (\alpha = (1+\sqrt{5})/2)
\end{equation*}
\begin{equation*}
\sum_{n = 0}^\infty \frac{2^{2n}}{\binom{{n}}{h}\binom{{2\left( {n + 1} \right)}}{{\left( {n + 1} \right)}}\left( {n + 1} \right)^2}
= \frac{\pi}{4}\,\frac{H_{h - 1/2} - H_{- 1/2}}{\sin (\pi h)},\quad h\in\mathbb C\setminus\mathbb Z,\; 2h\not\in\mathbb Z^{-},
\end{equation*}
and for non-negative integer $h$,
\begin{equation*}
\begin{split}
\sum_{n = h}^\infty \frac{{\binom{{2\left( {n - h} \right)}}{{n - h}}}}{{\binom{{n}}{h}\binom{{2\left( {n + 1} \right)}}{{n + 1}}}}\frac{{O_{n - h} }}{{\left( {n + 1} \right)^2 }} &= - \left( { - 1} \right)^h \frac{{\zeta \left( 2 \right) - H_h^{(2)} }}{{2\left( {h + 1} \right)\binom{{2\left( {h + 1} \right)}}{{h + 1}}}} + \left( { - 1} \right)^h \frac{{H_h + 2\ln 2}}{{\left( {h + 1} \right)\binom{{2\left( {h + 1} \right)}}{{h + 1}}}}{O_{h + 1}}\\
&\qquad\qquad - \sum_{n = 0}^{h - 1} {\frac{{\left( { - 1} \right)^{n - h} \binom{{h}}{n}}}{{\binom{{2\left( {h - n} \right)}}{{h - n}}\binom{{2\left( {n + 1} \right)}}{{n + 1}}}}\frac{{O_{n - h} }}{{\left( {n + 1} \right)^2 }}}.
\end{split}
\end{equation*}

\section{First new series associated with \eqref{psi_expr3} }

We firstly prove two results that are an immediate consequence of \eqref{psi_expr3} and which yield
some possibly new series for $\pi$. Two such series provide unexpected expressions involving the golden section $\alpha=(1+\sqrt{5})/2$.
We then evaluate several types of series involving the ratio of Catalan numbers and binomial coefficients in closed form.

\begin{theorem}\label{thm_series_pi}
For each integer $k\geq 2$ we have the expression
\begin{equation}
(k-2) \left (1 + \sum_{n=1}^\infty \frac{1}{n+1} \prod_{j=1}^n \frac{kj - (k-2)}{kj+1}\right ) = \pi \cot\left (\frac{\pi}{k}\right ).
\end{equation}
In particular, we have the curious expressions for $\pi$ in the form
\begin{equation}\label{pi_id1}
1 + \sum_{n=1}^\infty \frac{1}{n+1} \prod_{j=1}^n \frac{3j-1}{3j+1} = \frac{\pi}{\sqrt{3}},
\end{equation}
\begin{equation}\label{pi_id2}
1 + \sum_{n=1}^\infty \frac{2^n}{n+1} \prod_{j=1}^n \frac{2j-1}{4j+1} = \frac{\pi}{2},
\end{equation}
\begin{equation}\label{pi_id3}
1 + \sum_{n=1}^\infty \frac{1}{n+1} \prod_{j=1}^n \frac{5j-3}{5j+1} = \frac{\pi}{3} \sqrt{\frac{\alpha^3}{\sqrt{5}}},
\end{equation}
and
\begin{equation}\label{pi_id4}
1 + \sum_{n=1}^\infty \frac{2^n}{n+1} \prod_{j=1}^n \frac{3j-2}{6j+1} = \frac{\sqrt{3} \pi}{4}.
\end{equation}
\end{theorem}
\begin{proof}
Set $h=1-2z$ in \eqref{psi_expr3}. Then, for all $z\in\mathbb{C}$ with $0<\Re(z)<1$, we have
\begin{equation*}
\pi \cot (\pi z) = \psi(1-z) - \psi(z) = \sum_{n=0}^\infty \frac{(-1)^n}{n+1} \prod_{j=0}^n \left (\frac{1-z}{j+z} - 1\right)
\end{equation*}
and with $z=1/k,k\geq 2,$ the claim follows.
\end{proof}

\begin{theorem}\label{thm2_series_pi}
For all $z\in\mathbb{C}$ with $-1/2<\Re(z)<1/2$, we have
\begin{equation}\label{thm2_pi_id}
\sum_{n=0}^\infty \frac{2^{n+1}}{n+1} \prod_{j=0}^n \frac{2z-j}{2z-(2j+1)} = -\pi \tan (\pi z).
\end{equation}
In particular, when $z=1/3$ then \eqref{thm2_pi_id} gives \eqref{pi_id4},
and when $z=1/4$ then we get \eqref{pi_id2}. When $z=1/5$, then \eqref{thm2_pi_id} yields
\begin{equation}
1 + \sum_{n=1}^\infty \frac{2^n}{n+1} \prod_{j=1}^n \frac{5j-2}{10j+3} = \frac{3\pi}{4}\sqrt{\frac{\sqrt{5}}{\alpha^3}}.
\end{equation}
\end{theorem}
\begin{proof}
In \eqref{psi_expr3} make the replacements $z\mapsto 1/2-z$ and $h\mapsto 2z$, while keeping in mind that
\begin{equation*}
\psi \left(\frac{1}{2} + z \right ) - \psi \left(\frac{1}{2} - z \right ) = \pi \tan (\pi z).
\end{equation*}
\end{proof}

Since
\begin{equation*}
h (h - 1) \cdots (h - n) = \prod_{j = 0}^n (h - j) = (- 1)^{n + 1} \frac{\Gamma(n + 1 - h)}{\Gamma(- h)}
\end{equation*}
and
\begin{equation*}
z (z + 1) \cdots (z + n) = \prod_{j = 0}^n (z + j) = \frac{\Gamma(n + 1 + z)}{\Gamma(z)},
\end{equation*}
identity~\eqref{psi_expr3} can be written in terms of harmonic numbers via \eqref{Harm_psi} as
\begin{equation}\label{psi_expr4}
H_{z + h - 1} - H_{z - 1} =  - \frac{\Gamma (z)}{\Gamma (- h)} \sum_{n = 0}^\infty \frac{\Gamma (n + 1 - h)}{(n + 1)\Gamma (n + 1 + z)}.
\end{equation}

\begin{lemma}\label{lem.ho}
We have
\begin{gather}
H_{k - 1/2} = 2O_k - 2\ln 2, \\
H_{k - 1/2}^{(2)} = - 2\zeta \left( 2 \right) + 4O_k^{(2)}, \\
H_{- 1/2}^{(3)} = - 6\zeta \left( 3 \right), \\
H_{k - 1/2}^{(m + 1)} - H_{- 1/2}^{(m + 1)} = 2^{m + 1} O_k^{(m + 1)}.
\end{gather}
\end{lemma}

\begin{lemma}\label{lem.iyajmzy}
We have
\begin{gather}
\left( {r - 1/2} \right)! = \frac{{\sqrt \pi  }}{{2^{2r} }}\binom{{2r}}{r}r!,\quad r\in\mathbb C\setminus\mathbb Z^{-},\\
\left( { - r - 1/2} \right)! = \sqrt \pi  \frac{{\left( { - 1} \right)^r }}{{r!}}\frac{{2^{2r} }}{{\binom{{2r}}{r}}},\quad r\in\mathbb N_0,\\
\end{gather}
\end{lemma}

\begin{theorem}\label{Cat_series1}
We have
\begin{equation}
\sum_{n = 0}^\infty \frac{{C_n }}{{2^{2n} \left( {n + 1} \right)}} = 4\left( {1 - \ln 2} \right),
\end{equation}
\begin{equation}
\sum_{n = 0}^\infty \frac{{C_n }}{{2^{2n} \left( {n + 1} \right)\left( {n + 2} \right)}} = \frac{{10}}{3} - 4\ln 2,
\end{equation}
and more generally, if $0\ne z\in\mathbb C\setminus\mathbb Z^{-}$, $2z\not\in\mathbb Z^{-}$, then
\begin{equation}\label{eq.n0r1wj4}
\sum_{n = 0}^\infty \frac{C_n}{{2^{2n}\binom{{n + z}}{n}}} = -2z \left( {H_{z - 1} - H_{z - 1/2} } \right),
\end{equation}
where here and throughout this paper $C_j$ are Catalan numbers, defined for $j\in\mathbb N_0$ by
\begin{equation*}
C_j = \frac{\binom{{2j}}{j}}{j + 1}.
\end{equation*}
\end{theorem}
\begin{proof}
Set $h=1/2$ in~\eqref{psi_expr4}, express as factorials and arrange as
\begin{equation*}
- \sum_{n = 0}^\infty \frac{{\left( {n - 1/2} \right)!}}{{n!\left( {n + 1} \right)}}\frac{1}{z}\frac{{z!n!}}{{\left( {z + n} \right)!}}
= 2\sqrt \pi (H_{z - 1} - H_{z - 1/2}),
\end{equation*}
that is
\begin{equation}
 - \sum_{n = 0}^\infty \frac{{\sqrt \pi  }}{{2^{2n} }}\frac{{\binom{{2n}}{n}}}{{n + 1}}\frac{1}{z}\frac{1}{{\binom{{z + n}}{n}}}
= 2\sqrt \pi (H_{z - 1} - H_{z - 1/2}),
\end{equation}
and hence~\eqref{eq.n0r1wj4}.
\end{proof}

\begin{theorem}\label{Cat_series2}
We have
\begin{equation}\label{eq.nruaw64}
\sum_{n = 0}^\infty \frac{{O_{n + 1}}}{{(n + 1)(2n + 1)}} = \frac{\pi^2}{6},
\end{equation}
\begin{equation}\label{eq.osdi1i1}
\sum_{n = 0}^\infty \frac{{C_n }}{{2^{2n} }}\frac{H_{n + 1}}{n + 1} = 8 - \frac{{2\pi ^2 }}{3},
\end{equation}
and more generally, if $0\ne z\in\mathbb C\setminus\mathbb Z^{-}$, $2z\not\in\mathbb Z^{-}$, then
\begin{equation}\label{eq.ue5hg8p}
\sum_{n = 0}^\infty \frac{{C_n }}{{2^{2n} }}\frac{H_{n + z}}{\binom{{n + z}}{n}}
= 2 (H_{z - 1} - H_{z - 1/2})(1 - z H_z) + 2z \left(- H_{z - 1}^{(2)} + H_{z - 1/2}^{(2)} \right).
\end{equation}
\end{theorem}
\begin{proof}
Differentiate~\eqref{eq.n0r1wj4} with respect to $z$ to obtain
\begin{equation}
\sum_{n = 0}^\infty \frac{{C_n }}{{2^{2n} }}\frac{{\left( {H_{n + z} - H_z } \right)}}{{\binom{{n + z}}{n}}}
= 2H_{z - 1} - 2H_{z - 1/2} + 2z \left( - H_{z - 1}^{(2)} + H_{z - 1/2}^{(2)} \right);
\end{equation}
from which~\eqref{eq.ue5hg8p} follows upon a second use of~\eqref{eq.n0r1wj4}. Evaluation of~\eqref{eq.ue5hg8p}
at $z=1/2$ gives~\eqref{eq.nruaw64} while evaluation at $z=1$ gives~\eqref{eq.osdi1i1}.
\end{proof}

\begin{theorem}\label{Cat_series3}
We have
\begin{equation}
\sum_{n = 0}^\infty \frac{2n+1}{n+1} \frac{C_n}{2^{2n}} = 4\ln 2,
\end{equation}
\begin{equation}
\sum_{n = 0}^\infty \frac{2n+1}{(n+1)(n+2)} \frac{C_n}{2^{2n}} = 4\ln 2 - 2,
\end{equation}
and more generally, if $z\in\mathbb{C}$ with $z - 3/2 \not\in\mathbb Z^{-}$, then
\begin{equation}\label{eq.cat_series3}
\sum_{n = 0}^\infty \frac{2n+1}{2^{2n}} \frac{C_n}{\binom{{n + z}}{n}} = 2z \left( H_{z - 1} - H_{z - 3/2} \right).
\end{equation}
\end{theorem}
\begin{proof}
The proof is very similar to the proof of Theorem \ref{Cat_series1}. Set $h=-1/2$ in \eqref{psi_expr3} and use $\Gamma(1/2)=\sqrt{\pi}$.
\end{proof}

\begin{corollary}\label{Cat_series4}
If $z\in\mathbb{C}$ with $z - 3/2\not\in\mathbb Z^{-}$, then
\begin{equation}\label{eq.cat_series4}
\sum_{n = 0}^\infty \frac{n}{2^{2n}} \frac{C_n}{\binom{{n + z}}{n}} = 2z H_{z - 1} - z\left( H_{z - 1/2} + H_{z - 3/2} \right).
\end{equation}
\end{corollary}
\begin{proof}
Combine Theorems \ref{Cat_series1} and \ref{Cat_series3}.
\end{proof}

\begin{remark}
We have
\begin{equation}
\sum_{n = 0}^\infty \frac{n}{(n+1)} \frac{C_n}{2^{2n}} = \sum_{n = 0}^\infty \frac{2n+1}{(n+1)(n+2)} \frac{C_n}{2^{2n}} = 4\ln 2 - 2.
\end{equation}
\end{remark}

\begin{theorem}\label{Cat_series5}
If $z\in\mathbb{C}$ with $z - 3/2\not\in\mathbb Z^{-}$, then
\begin{equation}\label{eq.cat_series5}
\sum_{n = 0}^\infty \frac{2n+1}{2^{2n}} \frac{C_n H_{n+z}}{\binom{{n + z}}{n}} = 2(z H_z - 1) \left( H_{z - 1} - H_{z - 3/2} \right)
+ 2z \left( H_{z - 1}^{(2)} - H_{z - 3/2}^{(2)} \right).
\end{equation}
In particular,
\begin{equation}
\sum_{n = 0}^\infty \frac{2n+1}{n+1} \frac{C_n H_{n+1}}{2^{2n}} = \frac{2 \pi^2}{3},
\end{equation}
\begin{equation}
\sum_{n = 0}^\infty \frac{2n+1}{(n+2)(n+1)} \frac{C_n H_{n+2}}{2^{2n}} = \frac{2 \pi^2}{3} + 4 \ln 2 - 8,
\end{equation}
and
\begin{equation}\label{eq.nr7v570}
\sum_{n = 0}^\infty \frac{O_{n + 2}}{(n + 1)(2n + 3)} = 4 - 2\ln 2 - \frac{\pi^2}{6}.
\end{equation}
\end{theorem}
\begin{proof}
Differentiate \eqref{eq.cat_series3} with respect to $z$ and simplify. Identity~\eqref{eq.nr7v570} is obtained 
by setting $z=3/2$ in~\eqref{eq.cat_series5} and using Lemma~\ref{lem.ho}.
\end{proof}

\begin{theorem}\label{Cat_series6}
If $z\in\mathbb{C}$ with $z - 3/2\not\in\mathbb Z^{-}$, then
\begin{align}\label{eq.cat_series6}
\sum_{n = 0}^\infty \frac{n}{2^{2n}} \frac{C_n H_{n+z}}{\binom{{n + z}}{n}}
&= (z H_z - 1) \left( 2 H_{z - 1} - H_{z - 1/2} - H_{z - 3/2} \right) \nonumber \\
&\qquad + z \left( 2 H_{z - 1}^{(2)} - H_{z - 1/2}^{(2)} - H_{z - 3/2}^{(2)} \right).
\end{align}
In particular, at $z=1$ and at $z=3/2$ we obtain
\begin{equation}
\sum_{n = 1}^\infty \frac{n C_n}{2^{2n}}\,\frac{H_{n + 1}}{n + 1} = \frac{2\pi^2}{3} - 4
\end{equation}
and
\begin{equation}
\sum_{n = 1}^\infty \frac{n O_{n + 2}}{(n + 1)(2n + 1)(2n + 3)} = \frac{13}{4} - 2\ln 2 - \frac{\pi^2}{6}.
\end{equation}
\end{theorem}
\begin{proof}
Differentiate \eqref{eq.cat_series4} with respect to $z$ and simplify.
\end{proof}

Note that we used
\begin{equation*}
\binom{{n + 3/2}}{n} = \frac{(2n + 1)(2n + 3)}{2^{2n}\,3}\binom{{2n}}{n}.
\end{equation*}

\section{More new series associated with \eqref{psi_expr3} }

\begin{theorem}
We have
\begin{equation}\label{eq.pm4exox}
\sum_{n = 0}^\infty \frac{{2^{2n} }}{{\binom{{2\left( {n + 1} \right)}}{{n + 1}}\left( {n + 1} \right)^2 }} = \frac{\pi^2}{8},
\end{equation}
\begin{equation}\label{eq.gujh5lu}
\sum_{n = 0}^\infty \frac{1}{{\left( {2n + 1} \right)\left( {n + 1} \right)}} = 2\ln 2,
\end{equation}
and, more generally, for $h\not\in\mathbb Z$, $2h\not\in\mathbb Z^{-}$,
\begin{equation}\label{eq.giggzhv}
\sum_{n = 0}^\infty \frac{2^{2n}}{\binom{{n}}{h}\binom{{2\left( {n + 1} \right)}}{{\left( {n + 1} \right)}}\left( {n + 1} \right)^2}
= \frac{\pi }{4}\,\frac{H_{h - 1/2} - H_{- 1/2}}{\sin(\pi h)}.
\end{equation}
\end{theorem}
\begin{proof}
Set $z=1/2$ in~\eqref{psi_expr4} and express as factorials to obtain
\begin{equation*}
\sum_{n = 0}^\infty \frac{{\left( {n - h} \right)!h!}}{{n!}}\frac{{n!\sqrt \pi}}{{\left( {n + 1/2} \right)!}}\frac{1}{{n + 1}}
= h!\left( {- 1 - h} \right)!\left( {H_{- 1/2}  - H_{h - 1/2} } \right),
\end{equation*}
from which~\eqref{eq.giggzhv} follows upon using Lemma~\ref{lem.iyajmzy} and the Euler reflection formula:
\begin{equation*}
\left( { - r} \right)!\left( {r - 1} \right)! = \frac{\pi}{\sin(\pi r)},\quad r\not\in\mathbb Z.
\end{equation*}
Identities~\eqref{eq.pm4exox} and~\eqref{eq.gujh5lu} are special cases of~\eqref{eq.giggzhv} at $h=0$ and $h=1/2$.
Note that in deriving~\eqref{eq.pm4exox}, we used L'Hospital's rule:
\begin{equation*}
\begin{split}
\lim_{h\to 0}\frac{H_{h - 1/2} - H_{- 1/2}}{\sin(\pi h)} &= \lim_{h\to 0}\frac{{\psi \left( {h + 1/2} \right) - \psi \left( {1/2} \right)}}{\sin (\pi h)} \\
&= \lim_{h\to 0}\frac{\psi^{(1)} \left( {h + 1/2} \right)}{\pi \cos(\pi h)} = \frac{\pi}{2}.
\end{split}
\end{equation*}
\end{proof}

\begin{corollary}
If $h$ is a non-negative integer, then
\begin{equation}\label{eq.t7z1d41}
\sum_{n = h}^\infty \frac{{\binom{{2\left( {n - h} \right)}}{{n - h}}}}{{\binom{{n}}{h}\binom{{2\left( {n + 1} \right)}}{{\left( {n + 1} \right)}}\left( {n + 1} \right)^2 }} = (- 1)^h \frac{H_h + 2\ln 2}{\left( {h + 1} \right)\binom{{2\left( {h + 1} \right)}}{{h + 1}}}
- \sum_{n = 0}^{h - 1} \frac{(- 1)^{n - h} \binom{{h}}{n}}{\binom{{2\left( {h - n} \right)}}{{\left( {h - n} \right)}}\binom{{2\left( {n + 1} \right)}}{{n + 1}}\left( {n + 1} \right)^2 }.
\end{equation}
\end{corollary}
\begin{proof}
Write $h+1/2$ for $h$ in~\eqref{eq.giggzhv} and use the fact that if $n$ and $h$ are non-negative integers, then
\begin{equation}\label{eq.ii21j6e}
\binom{{n}}{{h + 1/2}} = \frac{{2^{2n + 2} }}{{\pi \left( {h + 1} \right)}}
\begin{cases}
 \binom{{n}}{h}\binom{{2\left( {n - h} \right)}}{{n - h}}^{ - 1} \binom{{2\left( {h + 1} \right)}}{{h + 1}}^{ - 1},&\text{if $n\ge h$;}  \\
 \\
 (-1)^{h-n}\binom{{2\left( {h - n} \right)}}{{h - n}}\binom{{h}}{n}^{ - 1} \binom{{2\left( {h + 1} \right)}}{{h + 1}}^{ - 1}&\text{if $n\le h$;}  \\
 \end{cases}
\end{equation}
since
\begin{equation*}
\binom{{n}}{{h + 1/2}} =  \begin{cases}
 \frac{{n!}}{{\left( {n + 1/2} \right)!\left( {n - h - 1/2} \right)!}}\,, &\text{if $n\ge h$;}\\
 \frac{{n!}}{{\left( {n + 1/2} \right)!\left( { - \left( {h - n} \right) - 1/2} \right)!}}\,, &\text{if $n\le h$.}\\
 \end{cases}
\end{equation*}
\end{proof}

\begin{theorem}
If $h$ is a non-negative integer, then
\begin{equation}\label{last_eq}
\begin{split}
\sum_{n = h}^\infty  {\frac{{\binom{{2\left( {n - h} \right)}}{{n - h}}}}{{\binom{{n}}{h}\binom{{2\left( {n + 1} \right)}}{{n + 1}}}}\frac{{O_{n - h} }}{{\left( {n + 1} \right)^2 }}}  &=  - \left( { - 1} \right)^h \frac{{\zeta \left( 2 \right) - H_h^{(2)} }}{{2\left( {h + 1} \right)\binom{{2\left( {h + 1} \right)}}{{h + 1}}}} + \left( { - 1} \right)^h \frac{{H_h  + 2\ln 2}}{{\left( {h + 1} \right)\binom{{2\left( {h + 1} \right)}}{{h + 1}}}}{O_{h + 1}}\\
&\qquad\qquad - \sum_{n = 0}^{h - 1} {\frac{{\left( { - 1} \right)^{n - h} \binom{{h}}{n}}}{{\binom{{2\left( {h - n} \right)}}{{h - n}}\binom{{2\left( {n + 1} \right)}}{{n + 1}}}}\frac{{O_{h - n} }}{{\left( {n + 1} \right)^2 }}} .
\end{split}
\end{equation}
\end{theorem}
\begin{proof}
Differentiate~\eqref{eq.giggzhv} with respect to $h$ to obtain
\begin{equation*}
\sum_{n = 0}^\infty \frac{2^{2n} \left( H_{n - h} - H_h \right)}{\binom{{n}}{h}\binom{{2(n + 1)}}{{n + 1}} (n + 1)^2}
= \frac{\pi^2 \cos(\pi h)}{4\sin^2 (\pi h)}\left( H_{h - 1/2} - H_{- 1/2} \right) - \frac{\pi}{4\sin(\pi h)}\left(\zeta (2) - H_{h - 1/2}^{(2)}\right).
\end{equation*}
Now write $h + 1/2$ for $h$ and use Lemma~\ref{lem.ho} and identities~\eqref{eq.t7z1d41} and~\eqref{eq.ii21j6e}.
\end{proof}

Note that when $h=0$ we get
\begin{equation}
\sum_{n=0}^\infty \frac{O_n}{(n+1)(2n+1)} = - \frac{\pi^2}{12} + 2 \ln 2,
\end{equation}
which in view of \eqref{eq.nruaw64} also gives
\begin{equation}
\sum_{n=0}^\infty \frac{1}{(n+1)(2n+1)^2} = \frac{\pi^2}{4} - 2 \ln 2.
\end{equation}
When $h=1$ we see that
\begin{equation}
\sum_{n=1}^\infty \frac{O_{n-1}}{(n+1)(2n+1)(2n-1)} = \frac{\pi^2}{36} - \frac{8}{9} \ln 2 + \frac{7}{18}.
\end{equation}

\begin{theorem}
If $m$ is a positive integer, then
\begin{equation}\label{eq.alvbeou}
\sum_{n = 0}^\infty \frac{2^{2n}}{(n + 1)\binom{{n + m}}{m}\binom{2(n + m)}{n + m}} 
= \frac{m}{4}\,\frac{3\,\zeta (2) - 4\,O_{m - 1}^{(2)}}{\binom{2(m - 1)}{m - 1}}.
\end{equation}
In particular,
\begin{equation}
\sum_{n = 0}^\infty \frac{2^{2n}}{(n + 1)^2 \binom{2(n + 1)}{n + 1}} = \frac{\pi^2}{8}.
\end{equation}
\end{theorem}
\begin{proof}
Arrange~\eqref{psi_expr4} as
\begin{equation*}
\sum_{n = 0}^\infty \frac{{\Gamma \left( {n + 1 - h} \right)}}{{\left( {n + 1} \right)\Gamma (n + 1 + z)}} 
= \Gamma (- h)\left(H_{z - 1} - H_{z + h - 1} \right)\frac{1}{\Gamma (z)}
\end{equation*}
and take the limit as $h$ approaches zero, using
\begin{equation*}
\lim_{h\to 0}\Gamma (- h)\left( H_{z - 1} - H_{z + h - 1} \right) = \lim_{h\to 0}\Gamma (- h)\left(\psi (z) - \psi (z + h) \right)
= \psi^{(1)} (z) = \zeta (2) - H_{z - 1}^{(2)},
\end{equation*}
to obtain
\begin{equation*}
\sum\limits_{n = 0}^\infty  {\frac{{n!}}{{\left( {n + 1} \right)\Gamma \left( {n + 1 + z} \right)}}}  = \frac{1}{{\Gamma \left( z \right)}}\left( {\zeta \left( 2 \right) - H_{z - 1}^{(2)} } \right).
\end{equation*}
Now replace $z$ with $z-1/2$ and use
\begin{equation}\label{eq.budvepa}
\Gamma \left( {r + 1/2} \right) = \binom{{r - 1/2}}{r}r!\sqrt \pi  ,
\end{equation}
to obtain
\begin{equation}\label{eq.z7cdm1p}
\sum_{n = 0}^\infty  {\frac{1}{{\left( {n + 1} \right)\binom{{n + z}}{z}\binom{{n + z - 1/2}}{{n + z}}}}} = \frac{z}{{\binom{{z - 3/2}}{{z - 1}}}}\left( {\zeta \left( 2 \right) - H_{z - 3/2}^{(2)} } \right),
\end{equation}
which is valid for $z\in\mathbb C\setminus\mathbb Z^{-}$ such that $2z\not\in\mathbb Z^{-}$, $z\ne 0$ and $z\ne 1/2$. Finally set $z=m$, a positive integer, in~\eqref{eq.z7cdm1p} and use
\begin{equation}\label{eq.vgfrd3}
\binom{{p - 1/2}}{p} = \frac{{\binom{{2p}}{p}}}{{2^{2p} }};
\end{equation}
which gives~\eqref{eq.alvbeou} on account of Lemma~\ref{lem.ho}.
\end{proof}

\begin{theorem}
If $m$ is a positive integer, then
\begin{equation}\label{eq.np2sv9u}
\begin{split}
&\sum_{n = 0}^\infty \frac{2^{2n}O_{n + m}}{(n + 1)\binom{n + m}{m} \binom{{2(n + m)}}{{n + m}}} \\
&\qquad = \frac{m}{4}\,\binom{2(m - 1)}{m - 1}^{-1} \left(7\,\zeta (3) - 8\,O_{m - 1}^{(3)} + O_{m - 1}(3\,\zeta (2) - 4\,O_{m - 1}^{(2)})\right).
\end{split}
\end{equation}
In particular,
\begin{equation}
\sum_{n = 0}^\infty \frac{2^{2n} O_{n + 1}}{\binom{2(n + 1)}{n + 1}(n + 1)^2} = \frac{7}{4}\,\zeta (3).
\end{equation}
\end{theorem}
\begin{proof}
Differentiate~\eqref{eq.z7cdm1p} with respect to $z$ to obtain
\begin{equation*}
\begin{split}
\sum_{n = 0}^\infty  {\frac{{H_{n + z - 1/2} }}{{\left( {n + 1} \right)\binom{{n + z - 1/2}}{{n + z}}\binom{{n + z}}{z}}}}  
&= \frac{z}{{\binom{{z - 3/2}}{{z - 1}}}}\left( \zeta \left( 2 \right) - H_{z - 3/2}^{(2)} \right)\left( {H_{z - 3/2} - H_{z - 1} } \right) \\
&\qquad + \frac{{2z}}{{\binom{{z - 3/2}}{{z - 1}}}}\left( {\zeta \left( 3 \right) - H_{z - 3/2}^{(3)} } \right) + \frac{{zH_{z - 1} }}{{\binom{{z - 3/2}}{{z - 1}}}}\left( {\zeta \left( 2 \right) - H_{z - 3/2}^{(2)} } \right)
\end{split}
\end{equation*}
and hence identity~\eqref{eq.np2sv9u} on account of Lemma~\ref{lem.ho} and identity~\eqref{eq.vgfrd3}.
\end{proof}

\begin{theorem}
If $m$ is a positive integer, then
\begin{equation}\label{eq.tu5lpx7}
\sum_{n = 0}^\infty  {\frac{{C_n}}{{\binom{{n + m}}{n}\binom{{2\left( {n + m} \right)}}{{n + m}}}}}  = \frac{m}{2}\,\frac{\left({H_{m - 1}  - 2\,O_{m - 1}  + 2\,\ln 2}\right)}{{\binom{{2\left( {m - 1} \right)}}{{m - 1}}}}.
\end{equation}
\end{theorem}
\begin{proof}
Set $h=1/2$ , write $z-1/2$ for $z$ in~\eqref{psi_expr4} and use~\eqref{eq.budvepa} to obtain
\begin{equation}\label{eq.pp9z2ev}
\sum_{n = 0}^\infty \frac{{\binom{{n - 1/2}}{n}}}{{\left( {n + 1} \right)\binom{{n + z - 1/2}}{{n + z}}\binom{{n + z}}{z}}} 
= \frac{{2z}}{{\binom{{z - 3/2}}{{z - 1}}}}\left( H_{z - 1} - H_{z - 3/2} \right),
\end{equation}
from which identity~\eqref{eq.tu5lpx7} follows ater setting $z=m$, a positive integer, and using~\eqref{eq.vgfrd3} and Lemma~\ref{lem.ho}.
\end{proof}

\begin{theorem}\label{thm_CO_xyz}
If $m$ is a positive integer, then
\begin{equation}\label{eq.h4z23f1}
\begin{split}
&\sum_{n = 0}^\infty \frac{C_n O_{n + m}}{\binom{{n + m}}{m} \binom{2(n + m)}{n + m}} \\ 
&\qquad = \frac{m}{4}\,\binom{2(m - 1)}{m - 1}^{- 1} \left( H_{m - 1}^{(2)} + 2\,\zeta (2) - 4\,O_{m - 1}^{(2)} 
+ 2\,O_{m - 1} \left( H_{m - 1} - 2\,O_{m - 1} + 2\ln 2 \right) \right).
\end{split}
\end{equation}
\end{theorem}
\begin{proof}
Differentiate~\eqref{eq.pp9z2ev} with respect to $z$, obtaining
\begin{equation*}
\begin{split}
\sum_{n = 0}^\infty  {\frac{{\binom{{n - 1/2}}{n}H_{n + z - 1/2} }}{{\left( {n + 1} \right)\binom{{n + z - 1/2}}{{n + z}}\binom{{n + z}}{z}}}}  &= \frac{{2z}}{{\binom{{z - 3/2}}{{z - 1}}}}\left( {H_{z - 1}^{(2)}  - H_{z - 3/2}^{(2)}  - \left( {H_{z - 1}  - H_{z - 3/2} } \right)^2 } \right)\\
&\qquad + \frac{{2zH_{z - 1} }}{{\binom{{z - 3/2}}{{z - 1}}}}\left( {H_{z - 1}  - H_{z - 3/2} } \right),
\end{split}
\end{equation*}
from which~\eqref{eq.h4z23f1} follows.
\end{proof}

\begin{remark}
When $m=1$ in Theorem \ref{thm_CO_xyz} we recover~\eqref{eq.nruaw64}.
\end{remark}



\end{document}